\documentclass[a4paper, 12pt,reqno]{amsart}
\usepackage{amsmath,amssymb,amsthm,enumerate}%,backref,cite
\allowdisplaybreaks
\theoremstyle{plain}
\newtheorem{theorem}{Theorem}
\newtheorem*{theorem*}{Theorem}
\newtheorem{lemma}{Lemma}
\newtheorem*{lemma*}{Lemma}
\theoremstyle{definition}

\newtheorem*{definition*}{Definition}
\theoremstyle{remark}
\newtheorem{remark}{Remark}
\newtheorem*{remark*}{Remark}

\begin{document}
\title[Certain functions and related problems]{On certain functions and related problems}
\author{Symon Serbenyuk}

\email{simon6@ukr.net}


% Key words

\begin{abstract}

The present article is devoted to the description of  further investigations of the author of this article. These investigations (in terms of various representations of real numbers) include  the generalized Salem functions and generalizations of the Gauss-Kuzmin problem.

\end{abstract}
\maketitle

Let $Q\equiv (q_k)$ be a fixed sequence of positive integers, $q_k>1$,  $\Theta_k$ be a sequence of the sets $\Theta_k\equiv\{0,1,\dots ,q_k-1\}$, and $\varepsilon_k\in\Theta_k$.

The Cantor series expansion 
\begin{equation}
\label{eq: Cantor series}
\frac{\varepsilon_1}{q_1}+\frac{\varepsilon_2}{q_1q_2}+\dots +\frac{\varepsilon_k}{q_1q_2\dots q_k}+\dots
\end{equation}
of $x\in [0,1]$,   first studied by G. Cantor in \cite{C1869}. It is easy to see that the Cantor series expansion is the  $q$-ary expansion
\begin{equation}
\label{eq: q-series}
\frac{\alpha_1}{q}+\frac{\alpha_2}{q^2}+\dots+\frac{\alpha_n}{q^n}+\dots
\end{equation}
of numbers  from the closed interval $[0,1]$ whenever the condition $q_k=q$ holds for all positive integers $k$. Here $q$ is a fixed positive integer, $q>1$, and $\alpha_n\in\{0,1,\dots , q-1\}$.

By $x=\Delta^Q _{\varepsilon_1\varepsilon_2\ldots\varepsilon_k\ldots}$  denote a number $x\in [0,1]$ represented by series \eqref{eq: Cantor series}. This notation is called \emph{the representation of $x$ by Cantor series \eqref{eq: Cantor series}.} Also, by $x=\Delta^q _{\alpha_1\alpha_2\ldots\alpha_k\ldots}$  denote a number $x\in [0,1]$ represented by series \eqref{eq: q-series}. This notation is called \emph{the $q$-ary representation of $x$.}

We note that certain numbers from $[0,1]$ have two different representations by Cantor series \eqref{eq: Cantor series}, i.e., 
$$
\Delta^Q _{\varepsilon_1\varepsilon_2\ldots\varepsilon_{m-1}\varepsilon_m000\ldots}=\Delta^Q _{\varepsilon_1\varepsilon_2\ldots\varepsilon_{m-1}[\varepsilon_m-1][q_{m+1}-1][q_{m+2}-1]\ldots}=\sum^{m} _{i=1}{\frac{\varepsilon_i}{q_1q_2\dots q_i}}.
$$
Such numbers are called \emph{$Q$-rational}. The other numbers in $[0,1]$ are called \emph{$Q$-irrational}.

Let $c_1,c_2,\dots, c_m$ be an
ordered tuple of integers such that $c_i\in\{0,1,\dots, q_i-~1\}$ for $i=\overline{1,m}$. 

\emph{A cylinder $\Delta^Q _{c_1c_2...c_m}$ of rank $m$ with base $c_1c_2\ldots c_m$} is a set of the form
$$
\Delta^Q _{c_1c_2...c_m}\equiv\{x: x=\Delta^Q _{c_1c_2...c_m\varepsilon_{m+1}\varepsilon_{m+2}\ldots\varepsilon_{m+k}\ldots}\}.
$$
That is any cylinder $\Delta^Q _{c_1c_2...c_m}$ is a closed interval of the form
$$
\left[\Delta^Q _{c_1c_2...c_m000}, \Delta^Q _{c_1c_2...c_m[q_{m+1}-1][q_{m+2}-1][q_{m+3}-1]...}\right].
$$

Define \emph{the shift operator $\sigma$ of expansion \eqref{eq: Cantor series}} by the rule
$$
\sigma(x)=\sigma\left(\Delta^Q _{\varepsilon_1\varepsilon_2\ldots\varepsilon_k\ldots}\right)=\sum^{\infty} _{k=2}{\frac{\varepsilon_k}{q_2q_3\dots q_k}}=q_1\Delta^{Q} _{0\varepsilon_2\ldots\varepsilon_k\ldots}.
$$

It is easy to see that 
\begin{equation*}
\label{eq: Cantor series 2}
\begin{split}
\sigma^n(x) &=\sigma^n\left(\Delta^Q _{\varepsilon_1\varepsilon_2\ldots\varepsilon_k\ldots}\right)\\
& =\sum^{\infty} _{k=n+1}{\frac{\varepsilon_k}{q_{n+1}q_{n+2}\dots q_k}}=q_1\dots q_n\Delta^{Q} _{\underbrace{0\ldots 0}_{n}\varepsilon_{n+1}\varepsilon_{n+2}\ldots}.
\end{split}
\end{equation*}

Therefore, 
\begin{equation}
\label{eq: Cantor series 3}
x=\sum^{n} _{i=1}{\frac{\varepsilon_i}{q_1q_2\dots q_i}}+\frac{1}{q_1q_2\dots q_n}\sigma^n(x).
\end{equation}

In \cite{S. Serbenyuk alternating Cantor series 2013}, the notion of  the generalized shift operator was introduced (in terms of alternating Cantor series (also, see \cite{preprint2013, slides2013})). One can describe this notion  in terms of positive Cantor series. 

Define the generalized shift operator $\sigma_m$  of expansion~\eqref{eq: Cantor series}  by the rule
\begin{equation*}
\label{eq: gen. shift op. for positive Cantor series}
\begin{split}
\sigma_m(x) &=\sigma_m\left(\Delta^Q _{\varepsilon_1\varepsilon_2\ldots\varepsilon_k\ldots}\right)=\sum^{m-1} _{i=1}{\frac{\varepsilon_i}{q_1q_2\cdots q_i}}+\sum^{m-1} _{j=m+1}{\frac{\varepsilon_j}{q_1q_2\cdots q_{m-1}q_{m+1}\cdots q_j}}\\
&\equiv \Delta^{Q^{'} _{m}} _{\varepsilon_1\varepsilon_2\ldots\varepsilon_{m-1}\varepsilon_{m+1}\ldots},
\end{split}
\end{equation*}
where $Q^{'} _{m}$ is the following sequence: $(q_1,q_2, \dots , q_{m-1}, q_{m+1}, q_{m+2}, \dots)$.

It is easy to see that
\begin{equation*}
\label{eq: gen. shift op. for positive Cantor series 1}
\begin{split}
\sigma_m(x) &=\sigma_m\left(\Delta^Q _{\varepsilon_1\varepsilon_2\ldots\varepsilon_k\ldots}\right)=\sum^{m-1} _{i=1}{\frac{\varepsilon_i}{q_1q_2\cdots q_i}}+\sum^{m-1} _{j=m+1}{\frac{\varepsilon_j}{q_1q_2\cdots q_{m-1}q_{m+1}\cdots q_j}}\\
&= xq_m-\frac{\varepsilon_m}{q_1q_2\cdots q_{m-1}}-(q_m-1)\sum^{m-1} _{i=1}{\frac{\varepsilon_i}{q_1q_2\cdots q_i}}.
\end{split}
\end{equation*}

Note that $\sigma_m=\sigma$ whenever $m=1$.

\begin{lemma}
The mapping $\sigma_m$ has the following properties:
\begin{enumerate}
\item  The mapping $\sigma_m$ is continuous at each point of the interval $\left(\inf\Delta^Q _{c_1c_2...c_m}, \sup\Delta^Q _{c_1c_2...c_m}\right)$. The endpoints of $\Delta^Q _{c_1c_2...c_m}$ are  points of discontinuity of the mapping.

\item The mapping $\sigma_m$ has a derivative almost everywhere (with respect to
the Lebesgue measure). If the mapping  has a derivative at the point $x=\Delta^Q _{\varepsilon_1\varepsilon_2...\varepsilon_k...}$, then $\left(\sigma_m\right)^{'}=q_m$.

\item The derivative of $\sigma_m$ does not exist at a Q-rational point 
$$
x=\Delta^Q _{\varepsilon_1\varepsilon_2...\varepsilon_n000...}=\Delta^Q _{\varepsilon_1\varepsilon_2...\varepsilon_{n-1}[\varepsilon_n-1][q_{n+1}-1][q_{n+2}-1]...}
$$ whenever $n=m$.

\item $\sigma \circ \underbrace{\sigma_2\circ \ldots \circ \sigma_2(x)}_{m}=\sigma^{m+1}(x)$.

\item Suppose $(k_n)$ is an arbitrary increasing sequence of positive integers.  Then
$$
\sigma^{k_n-1} \circ \sigma_{k_n} \circ \sigma_{k_{n-1}} \circ \ldots \circ \sigma_{k_1}(x)=\sigma^{k_n+n-1} (x).
$$
\end{enumerate}
\end{lemma}

In the next articles of the author of this paper, the notion of the generalized shift operator will be investigated in more detail and applyed by the author of the present article  in terms of various representation of real numbers (e.g., positive and alternating Cantor series and their generalizations, as well as Luroth, Engel series, etc., various continued fractions).
%%%%%%%%%%%%%%%%%%%%%%%%%%%%%%%%%%%%%%%%%%%%%%%%%%%%%%%%%%%%%%%%%%%%%%%%

Let us consider certain applications of the generalized shift operator. One can model generalizations of the Gauss-Kuzmin problem and generalizations of the Salem function.

\section{Generalizations of the Gauss-Kuzmin problem}

The Gauss-Kuzmin problem is to calculate the limit
$$\lim_{n\to\infty}{\lambda\left(E_n(x)\right)},
$$
where $\lambda(\cdot)$ is the Lebesgue measure of a set and  the set $E_n(x)$ is a set of the form
$$
E_n=\left\{z: \sigma^n (z)<x\right\}.
$$
Here $z=\Delta _{i_1i_2...i_k...}$, i.e., $\Delta _{i_1i_2...i_k...}$ is a certain representation of real numbers, $\sigma$ is the shift operator. 

Ganeralizations of the Gauss-Kuzmin problem are to calculate the limit
$$
\lim_{k\to\infty}{\lambda\left(\tilde E^{} _{n_k}(x)\right)},
$$
for  sets of the following forms:
\begin{itemize}
\item 
$$
\tilde E^{} _{n_k}(x)=\left\{z: \sigma_{n_k}\circ \sigma_{n_{k-1}}\circ \ldots \circ \sigma_{n_1}(z)<x\right\}
$$
including (here $(n_k)$ is a certain fixed sequence of positive integers) the cases when $(n_k)$ is a constant sequence. 

\item the set $\tilde E^{} _{n_k}(x)$ under the condition that $n_k=\psi(k)$, where $\psi$ is a certain function of the positive integer argument. 

\item
$$
\tilde E^{} _{n_k}(x)=\left\{z: \underbrace{\sigma_{n_k}\circ \sigma_{n_{k-1}}\circ \ldots \circ \sigma_{n_1}}_{\varphi(m,k,c)}(z)<x\right\},
$$
where $\varphi$ is a certain function and  $m,c$ are some parameters (if applicable). That is, for example, 
$$
\tilde E^{} _{n_k}(x)=\left\{z: \underbrace{\sigma_{m}\circ \sigma_{m}\circ \ldots \circ \sigma_{m}}_{k}(z)<x\right\},
$$
where $k>c$ and $c$ is a fixed positive integer, 
or
$$
\tilde E^{} _{n_k}(x)=\left\{z: \underbrace{\sigma_{m}\circ \sigma_{m}\circ \ldots \circ \sigma_{m}}_{k}(z)<x\right\},
$$
where $k \equiv 1 (\mod c) $ and $c>1$ is a fixed positive integer. 

\item In the general case,
$$
\tilde E^{} _{n_k}(x)=\left\{z: \underbrace{\sigma_{\psi(\varphi(m,k,c))}\circ  \ldots \circ \sigma_{\psi(1)}}_{\varphi(m,k,c)}(z)<x\right\},
$$
\end{itemize}

In addition, one can formulate  such problems in terms of the shift operator. For example, one can formulate the Gauss-Kuzmin problem for  the following sets:
$$
\tilde E^{} _{n_k}(z)=\left\{z: \sigma^{n_k}(z)<\sigma^{k_0}(z)\right\},
$$
where $k_0$, $(n_k)$ are a fixed number and a fixed sequence.
$$
\tilde E^{} _{n_k}(x)=\left\{z: \sigma^{n_k}(z)<\sigma^{k_0}(x)\right\}.
$$
$$
\tilde E^{} _{n}(x)=\left\{z: \sigma^{\psi(n)}(z)<x\right\},
$$
where $\psi(n)$ is a certain function of the positive integer argument.

In addition, 
$$
\tilde E^{} _{n}(z)=\left\{z: \sigma^{\psi(n)}(z)<\sigma^{\varphi(n)}(z)\right\},
$$
$$
\tilde E^{} _{n}(x)=\left\{z: \sigma^{\psi(n)}(z)<\sigma^{\varphi(n)}(x)\right\},
$$
where $\psi, \varphi$  are certain functions of the positive integer arguments.

It is easy to see that similar problems can be formulated for the case of the generalized shift operator.

In next articles of the author of the present article, such problems will be considered by the author of this article in terms of various  numeral systems (with a finite or infinite alphabet, with a constant or variable alphabet, positive, alternating, and sign-variable expansions, etc.).

%%%%%%%%%%%%%%%%%%%%%%%%%%%%%%%%%%%%

\section{Generalizations of the Salem function}
%%%%%%%%%%%%%%%%%%%%%%%%%%%%%%%%%%%%

In \cite{Salem1943}, Salem modeled the function 
$$
s(x)=s\left(\Delta^2 _{\alpha_1\alpha_2...\alpha_n...}\right)=\beta_{\alpha_1}+ \sum^{\infty} _{n=2} {\left(\beta_{\alpha_n}\prod^{n-1} _{i=1}{q_i}\right)}=y=\Delta^{Q_2} _{\alpha_1\alpha_2...\alpha_n...},
$$
where $q_0>0$, $q_1>0$, and $q_0+q_1=1$. This function is a singular function. However,  
generalizations of the Salem function can be non-differentiable functions or do not have the derivative on a certain set.

Note that certain Salem function generalizations are considered  in \cite{S. Serbenyuk systemy rivnyan 2-2, Symon2015, Symon2017, Symon2019} as well. 

Suppose $(n_k)$ is a fixed sequence of positive integers such that $n_i\ne n_j$ for $i\ne j$ and such that  for any $n\in\mathbb N$ there exists a number $k_0$ for which the condition $n_{k_0}=n$ holds.

Let us consider the following infinite system of functional equations
$$
f\left(\sigma_{n_{k-1}}\circ \sigma_{n_{k-2}}\circ \ldots \circ \sigma_{n_1}(x)\right)=\beta_{\alpha_{n_k}, n_k}+p_{\alpha_{n_k}, n_k}f\left(\sigma_{n_{k}}\circ \sigma_{n_{k-1}}\circ \ldots \circ \sigma_{n_1}(x)\right),
$$
where $k=1,2,\dots$, $\sigma_0(x)=x$, and $x$ represented in terms of a certan given numeral system, i.e., $x=\Delta_{\alpha_1\alpha_2...\alpha_k...}$ and $\alpha_n\in\{0,1,\dots , m_n\}$ for all positive integers $n$. Also, here $P=||p_{i,n}||$ is a fixed matrix, where  $i=\overline{0,m_n}$, $m_n\in \mathbb N\cup \{0\}$, $n=1,2,\dots$, and for elements  $p_{i,n}$ of $P$ the following system of conditions is true: 

$$
\left\{
\begin{aligned}
\label{eq: tilde Q 1}
1^{\circ}.~~~~~~~~~~~~~~~~~~~~~~~~~~~~~~~~~~~~~~~~~~~~~~ ~~~~~ p_{i,n}\in (-1,1)\\
2^{\circ}.  ~~~~~~~~~~~~~~~~~~~~~~~~~~~~~~~~~~~~~~~~~~\forall n \in \mathbb N: \sum^{m_n}_{i=0} {p_{i,n}}=1\\
3^{\circ}.  ~~~~~~~~~~~~~~~~~~~~~~~~\forall (i_n), i_n \in \mathbb N \cup \{0\}: \prod^{\infty} _{n=1} {|p_{i_n,n}|}=0\\
4^{\circ}.~~~~~~~~~~~~~~~\forall  i_n \in \mathbb N: 0=\beta_{0,n}<\beta_{i_n,n}=\sum^{i_n-1} _{i=0} {p_{i,n}}<1.\\
\end{aligned}
\right.
$$

In the next articles of the author of this paper, properties of solutions of the last system of functional equations will be investigated for the cases of various numeral systems (with a finite or infinite alphabet, with a constant or variable alphabet, positive, alternating, and sign-variable expansions, etc.). 

Now one can begin this investigation with the case of the $q$-ary representation of real numbers. 

\begin{theorem}
Let $P_q=\{p_0,p_1,\dots , p_{q-1}\}$ be a fixed tuple of real numbers such that $p_i\in (-1,1)$, where $i=\overline{0,q-1}$, $\sum_i {p_i}=1$, and $0=\beta_0<\beta_i=\sum^{i-1} _{j=0}{p_j}<1$ for all $i\ne 0$. Then the system of functional equations
\begin{equation}
\label{eq: system-q}
f\left(\sigma_{n_{k-1}}\circ \sigma_{n_{k-2}}\circ \ldots \circ \sigma_{n_1}(x)\right)=\beta_{\alpha_{n_k}, n_k}+p_{\alpha_{n_k}, n_k}f\left(\sigma_{n_{k}}\circ \sigma_{n_{k-1}}\circ \ldots \circ \sigma_{n_1}(x)\right),
\end{equation}
where $x=\Delta^q _{\alpha_1\alpha_2...\alpha_k...}$, has the unique solution
$$
g(x)=\beta_{\alpha_{n_1}}+\sum^{\infty} _{k=2}{\left(\beta_{\alpha_{n_k}}\prod^{k-1} _{j=1}{p_{\alpha_{n_j}}}\right)}
$$
in the class of determined and bounded on $[0, 1]$ functions.
\end{theorem}
\begin{proof}
Since the function $g$ is a determined on $[0,1]$ function, using system~\eqref{eq: system-q},  we get 
$$
g(x)=\beta_{\alpha_{n_1}}+p_{\alpha_{n_1}}g(\sigma_{n_1}(x))
$$
$$
=\beta_{\alpha_{n_1}}+p_{\alpha_{n_1}}(\beta_{\alpha_{n_2}}+p_{\alpha_{n_2}}g(\sigma_{n_2}\circ\sigma_{n_1}(x)))=\dots
$$
$$
\dots =\beta_{\alpha_{n_1}}+\beta_{\alpha_{n_2}}p_{\alpha_{n_1}}+\beta_{\alpha_{n_3}}p_{\alpha_{n_1}}p_{\alpha_{n_2}}+\dots +\beta_{\alpha_{n_k}}\prod^{k-1} _{j=1}{p_{\alpha_{n_j}}}+\left(\prod^{k} _{t=1}{p_{\alpha_{n_t}}}\right)g(\sigma_{n_k}\circ \dots \circ \sigma_{n_2}\circ \sigma_{n_1}(x)).
$$

So,
$$
g(x)=\beta_{\alpha_{n_1}}+\sum^{\infty} _{k=2}{\left(\beta_{\alpha_{n_k}}\prod^{k-1} _{j=1}{p_{\alpha_{n_j}}}\right)}
$$
since $g$ is a  determined and bounded on $[0,1]$ function and 
$$
\lim_{k\to\infty}{g(\sigma_{n_k}\circ \dots \circ \sigma_{n_2}\circ \sigma_{n_1}(x))\prod^{k} _{t=1}{p_{\alpha_{n_t}}}}=0,
$$
where
$$
\prod^{k} _{t=1}{p_{\alpha_{n_t}}}\le \left( \max_{0\le i\le q-1}{p_i}\right)^k\to 0, ~~~ k\to \infty.
$$
\end{proof}
 
\begin{theorem}
The function $g$ is continuous at $q$-irrational points of $[0,1]$. The set of all points of discontinuities of the function $g$ is a countable, finite, or empty set. It  depends on a sequence $(n_k)$.
\end{theorem}

\begin{theorem}
Lebesgue integral of the function $g$ can be calculated by the
formula
$$
\int^1 _0 {g(x)dx}=\frac{\beta_1+\beta_2+\dots +\beta_{q-1}}{q-1}.
$$
\end{theorem}
\begin{remark}
 It can   be interesting to consider the case when $(n_k)$ is an arbitrary fixed  sequence (finite or infinite) of positive integers. Then  the function $g$ can be a constant function, a linear function,  or a function having pathological (complicated) structure, etc. It depends on  $(n_k)$. Such problems will be investigated in the next papers of the author of this article.
\end{remark}
To be continue...

\end{document}